\documentclass[12pt]{article}
\usepackage{amsmath,amssymb,amsthm,comment}

\newtheorem{theorem}{Theorem}
\newtheorem{thmy}{Theorem}

\newtheorem{lemma}[theorem]{Lemma}

\def\barr{\begin{array}}
\def\earr{\end{array}}

\title{On the solvability of a finite group by the sum of subgroup orders}
\author{Marius T\u arn\u auceanu}
\date{June 17, 2020}

\begin{document}

\maketitle

\begin{abstract}
Let $G$ be a finite group and $\sigma_1(G)=\frac{1}{|G|}\sum_{H\leq G}\,|H|$. Under some restrictions on the number of conjugacy classes of (non-normal) maximal subgroups of $G$, we prove that if $\sigma_1(G)<\frac{117}{20}\,$, then $G$ is solvable. This partially solves an open problem posed in \cite{9}.
\end{abstract}

{\small
\noindent
{\bf MSC2000\,:} Primary 20D60; Secondary 20D10, 20F16, 20F17.

\noindent
{\bf Key words\,:} subgroup orders, solvable groups.}

\section{Introduction}

Given a finite group $G$, we consider the function
$$\sigma_1(G)=\frac{1}{|G|}\sum_{H\leq G}\,|H|$$studied in our previous papers \cite{6,9}. Recall some basic properties of $\sigma_1$:
\begin{itemize}
\item[-] if $G$ is cyclic of order $n$ and $\sigma(n)$ denotes the sum of all divisors of $n$, then $\sigma_1(G)=\frac{\sigma(n)}{n}$\,;
\item[-] $\sigma_1$ is multiplicative, i.e. if $G_i$, $i=1,2,\dots,m$, are finite groups of coprime orders, then $\sigma_1(\prod_{i=1}^m G_i)=\prod_{i=1}^m \sigma_1(G_i)$;
\item[-] $\sigma_1(G)\geq\sigma_1(G/H)+\frac{1}{(G:H)}\,(\sigma_1(H)-1)\geq\sigma_1(G/H)$, for all $H\unlhd G$.
\end{itemize}

Let $k(G)$ and $k'(G)$ be the numbers of conjugacy classes of maximal subgroups of $G$ and of non-normal maximal subgroups of $G$, respectively. The starting point for our discussion is given by the open problem in \cite{9}, which asks to study whether there is a constant $c\in(2,\infty)$ such that $\sigma_1(G)<c$ implies the solvability of $G$. In the current note, we will show that if $k(G)\leq 3$ or $k'(G)\neq 3$, then such a constant is $\frac{117}{20}$\,.
\bigskip

For the proof of our results, we need the following three theorems from \cite{3} (see Theorems 2 and 1, respectively) and \cite{7} (see Exercise 7 of Section 10.5).

\begin{thmy}
A finite group $G$ with $k'(G)\leq 2$ is always solvable. In parti\-cu\-lar, a finite group $G$ with $k(G)\leq 2$ is always solvable.
\end{thmy}

\begin{thmy}
A finite group $G$ with $k(G)=3$ is non-solvable if and only if either $G/\Phi(G)\cong {\rm PSL}(2,7)$ or $G/\Phi(G)\cong {\rm PSL}(2,2^p)$, where $p$ is a prime.
\end{thmy}

\begin{thmy}
A finite group with an abelian maximal subgroup is always solvable.
\end{thmy}

We note that similar problems for some other functions related to the structure of a finite group $G$, for example for the function $\psi(G)=\sum_{x\in G}o(x)$ (where $o(x)$ denotes the order of the element $x$), have been recently investigated by many authors (see e.g. \cite{1,2,5}).
\bigskip

Most of our notation is standard and will not be repeated here. Basic definitions and results on groups can be found in \cite{7}. For subgroup lattice concepts we refer the reader to \cite{8}.

\section{Main results}

We start with an easy but important lemma.

\begin{lemma}
    Let $G$ be a finite group and $[M]$ be a conjugacy class of non-normal maximal subgroups of $G$. Then
$$\sum_{H\in[M]}|H|=|G|.$$
\end{lemma}

\begin{proof}
Since $M\subseteq N_G(M)$ and $M$ is not normal, we have $N_G(M)=M$. Therefore
$$|[M]|=(G:N_G(M))=(G:M),$$which leads to\newpage
$$\sum_{H\in[M]}|H|=|M||[M]|=|M|(G:M)=|G|,$$as desired.\qedhere
\end{proof}

We are now able to prove our first main result.

\begin{theorem}\label{th:C1}
    Let $G$ be a finite group with $k'(G)\neq 3$. If $\sigma_1(G)<\frac{117}{20}$, then $G$ is solvable.
\end{theorem}

\begin{proof}
For $k'(G)\leq 2$ the conclusion follows by Theorem A.

Assume that $k'(G)\geq 4$ and $\sigma_1(G)<\frac{117}{20}\,$, but $G$ is not solvable. Then it has no cyclic maximal subgroup by Theorem C. Let $[M_i]$, $i=1,2,3,4$, be four distinct conjugacy classes of non-normal maximal subgroups of $G$. We infer that
$$\sigma_1(G)\geq\frac{1}{|G|}\left(|G|+\sum_{i=1}^4\sum_{H\in[M_i]}|H|+\!\!\!\sum_{H\leq G,\, H=\text{\,cyclic}}\!\!\!|H|\right).$$From Lemma 1 we have $$\sum_{i=1}^4\sum_{H\in[M_i]}|H|=\sum_{i=1}^4 |G|=4|G|.$$Also, Theorem 2 of \cite{3} shows that
$$\sum_{H\leq G,\, H=\text{\,cyclic}}\!\!\!|H|=\sum_{a\in G}\frac{o(a)}{\phi(o(a))}\geq\sum_{a\in G}1=|G|\,.$$Then
$$\sigma_1(G)\geq\frac{1}{|G|}\left(|G|+4|G|+|G|\right)=6>\frac{117}{20}\,,$$a contradiction.\qedhere
\end{proof}

Next we will focus on proving our second main result. The following lemma will be helpful to us.

\begin{lemma}
    We have:
\begin{itemize}
\item[{\rm a)}] $\sigma_1({\rm PSL}(2,7))>\frac{117}{20}\,$;
\item[{\rm b)}] $\sigma_1({\rm PSL}(2,2^p))\geq\frac{117}{20}\,$, and the equality holds if and only if $p=2$.\newpage
\end{itemize}
\end{lemma}

\begin{proof}
\noindent\begin{itemize}
\item[{\rm a)}] By using GAP, we get $\sigma_1({\rm PSL}(2,7))=\frac{1499}{168}>\frac{117}{20}\,$, as desired.
\item[{\rm b)}] Assume first that $p\geq 3$ and let $q=2^p$. Then, by \cite{4}, ${\rm PSL}(2,q)$ has:
\begin{itemize}
\item[-] one subgroup of order $1$, namely the trivial subgroup;
\item[-] one subgroup of order $q^3-q$, namely ${\rm PSL}(2,q)$;
\item[-] three conjugacy classes of maximal subgroups $[M_i]$, $i=1,2,3$;
\item[-] $\frac{q(q+1)}{2}$ cyclic subgroups of order $m$, for every divisor $m\neq 1$ of $q-1$;
\item[-] $\frac{q(q-1)}{2}$ cyclic subgroups of order $m$, for every divisor $m\neq 1$ of $q+1$;
\item[-] $(q+1)\displaystyle\binom{p}{i}_{\!\!2}$\!\! elementary abelian subgroups of order $2^i$, for every $i=1,2,...,p$, where $$\displaystyle\binom{p}{i}_{\!\!2}=\frac{(2^p-1)\cdots(2-1)}{(2^i-1)\cdots(2-1)(2^{p-i}-1)\cdots(2-1)}$$is the Gaussian binomial coefficient.
\end{itemize}Note that every $M_i$ is non-normal because ${\rm PSL}(2,q)$ is a simple group. Thus, by Lemma 1, we have
$$\sum_{H\in[M_i]}|H|=|{\rm PSL}(2,q)|=\!q^3-q.$$It follows that
$$\sigma_1({\rm PSL}(2,q))\geq\frac{1}{q^3-q}\left[1+4(q^3-q)+\frac{q(q+1)}{2}\!\!\!\sum_{m|q-1,\, m\neq 1}\!\!\!m\,+\right.$$ $$\left.\hspace{-40mm}+\,\frac{q(q-1)}{2}\!\!\!\sum_{m|q+1,\, m\neq 1}\!\!\!m+(q+1)\sum_{i=1}^p\binom{p}{i}_{\!\!2}2^i\right]$$
$$\hspace{20mm}\geq\frac{1}{q^3-q}\left[1+4(q^3-q)+\frac{q(q+1)}{2}\,(q-1)\,+\right.$$ $$\left.\hspace{-14mm}+\,\frac{q(q-1)}{2}\,(q+1)+(q+1)\left(2\binom{p}{1}_{\!\!2}+4\binom{p}{2}_{\!\!2}+8\binom{p}{3}_{\!\!2}\right)\right]$$
$$\hspace{7mm}=5+\frac{1+(q+1)(q-1)\displaystyle\frac{q^2+8q+22}{21}}{q^3-q}$$
$$>5+\frac{(q+1)(q-1)\displaystyle\frac{q^2+8q+22}{21}}{q^3-q}$$
$$\hspace{5mm}=5+\frac{q^2+8q+22}{21q}>\frac{117}{20} \mbox{ for } q\geq 8,$$as desired.

Assume now that $p=2$. Then ${\rm PSL}(2,2^p)\cong A_5$ and we can easily check that $\sigma_1(A_5)=\frac{117}{20}\,$, completing the proof. \qedhere
\end{itemize}
\end{proof}

\begin{theorem}\label{th:C1}
    Let $G$ be a finite group with $k(G)=3$. If $\sigma_1(G)<\frac{117}{20}$, then $G$ is solvable.
\end{theorem}

\begin{proof}
Assume that the statement is false and let $G$ be a counterexample of minimal order. Then Theorem B leads to $$G/\Phi(G)\cong {\rm PSL}(2,7) \mbox{ or } G/\Phi(G)\cong {\rm PSL}(2,2^p), \mbox{ where } p \mbox{ is a prime.}$$ If $\Phi(G)\neq 1$, then
$$\sigma_1(G/\Phi(G))\leq\sigma_1(G)<\frac{117}{20}$$implies that $G/\Phi(G)$ is solvable by the minimality of $|G|$. Since $\Phi(G)$ is nilpotent, and consequently solvable, it follows that $G$ is also solvable, contradicting our assumption. Thus $\Phi(G)=1$, that is $$G\cong {\rm PSL}(2,7) \mbox{ or } G\cong {\rm PSL}(2,2^p) \mbox{ for a certain prime } p,$$and Lemma 3 implies that $\sigma_1(G)\geq\frac{117}{20}\,$, a contradiction. \qedhere
\end{proof}

Using Theorem B and Lemma 3, we also infer the following characterization of $A_5$.

\begin{theorem}\label{th:C1}
    Let $G$ be a non-solvable finite group with $k(G)=3$. If $\sigma_1(G)=\frac{117}{20}$, then $G\cong A_5$.
\end{theorem}

\begin{proof}
Under our hypotheses, we have $G/\Phi(G)\cong {\rm PSL}(2,7)$ or $G/\Phi(G)\cong {\rm PSL}(2,2^p)$, where $p$ is a prime, by Theorem B.

If $p>2$, then from Lemma 3 it follows that $\sigma_1(G/\Phi(G))>\frac{117}{20}\,$. Therefore\newpage $$\frac{117}{20}=\sigma_1(G)\geq\sigma_1(G/\Phi(G))>\frac{117}{20}\,,$$a contradiction. Thus $p=2$, that is $G/\Phi(G)\cong A_5$, and so $\sigma_1(G)=\sigma_1(G/\Phi(G))$. On the other hand, we know that $$\sigma_1(G)\geq\sigma_1(G/\Phi(G))+\frac{1}{(G:\Phi(G))}\,(\sigma_1(\Phi(G))-1).$$Then $\sigma_1(\Phi(G))=1$, i.e. $\Phi(G)$ is trivial. Consequently, $G\cong A_5$, completing the proof. \qedhere
\end{proof}

Finally, we formulate a natural problem concerning our study.
\bigskip

\noindent{\bf Open problem.} Let $G$ be an \textit{arbitrary} finite group with $\sigma_1(G)<\frac{117}{20}\,$. Is it true that $G$ is solvable?
\bigskip

Note that if the condition $\sigma_1(G)<\frac{117}{20}\,$ does not imply the solvability of $G$, then a counterexample $G$ of minimal order would be a just non-solvable group with $k'(G)=3$ by Theorem 2 and $k(G)\geq 4$ by Theorem 4. Thus $G$ would contain at least one normal maximal subgroup. Also, $G$ would be a Fitting-free group, that is it has no non-trivial solvable normal subgroup, or equivalently it has no non-trivial abelian normal subgroup.

\vspace*{5ex}\small

\hfill
\begin{minipage}[t]{5cm}
Marius T\u arn\u auceanu \\
Faculty of  Mathematics \\
``Al.I. Cuza'' University \\
Ia\c si, Romania \\
e-mail: {\tt tarnauc@uaic.ro}
\end{minipage}

\end{document}